\documentclass[a4paper,12pt]{article}

\usepackage{times,epsfig,wrapfig,amsmath,amsfonts,amssymb,amsthm,latexsym}
\usepackage{algorithm,algorithmic}

\newcommand{\E}{\ensuremath{\mathbb{E}}}

\newcommand{\var}{\mbox{Var}}
\def \trans {^{\top}}

\newtheorem{theorem}{Theorem}[section]
\newtheorem{lemma}[theorem]{Lemma}

\newtheorem{definition}[theorem]{Definition}
\newtheorem{claim}[theorem]{Claim}

\def\vtil{\tilde{v}}

\newcommand{\ignore}[1]{}

\DeclareMathOperator{\clip}{clip}

\def \sumt {\sum_{t\in[T]}}

\def \tr {\trans}
\newcommand{\abs}[1]{\Big| #1 \Big|}

\def \Rnxn {\mathbb{R}^{n\times{}n}}
\def \sumim {\sum_{i=1}^m}

\def \sumtT {\sum_{t=1}^T}

\def \ballf {\mathbb{B}_F}

\def \ident {\textbf{I}}

\def \sdpball {\mathbb{S}_{+}}

\usepackage{amsmath}
\usepackage{algorithmic}
\usepackage{algorithm}
\usepackage{amsthm}
\usepackage{amsfonts}
\usepackage{hyperref}
\usepackage{times}
\usepackage{color}

\title{Almost Optimal Sublinear Time Algorithm for Semidefinite Programming}

\author{Dan Garber\\  
\small{Technion - Israel Inst. of Tech.} \\
\small{dangar@cs.technion.ac.il} 
\and Elad Hazan \\
\small{Technion - Israel Inst. of Tech.} \\ 
\small{ehazan@ie.technion.ac.il}}

\begin{document}
\maketitle

\begin{abstract}
\noindent We present an algorithm for approximating semidefinite programs with running time that is sublinear in the number of entries in the semidefinite instance. We also present lower bounds that show our algorithm to have a nearly optimal running time \footnote{This work is a continuation and improvement of the sublinear SDP algorithm in \cite{DBLP:conf/nips/GarberH11}.}.

\end{abstract}

\section{Introduction}
We consider the following problem known as semidefinite programming
\begin{eqnarray}\label{SDP}
\textrm{Find}\quad    X\succeq{0} &   &    \\ \nonumber
\textrm{subject to}\quad  A_i\bullet{}X    &\geq & b_i  \quad i = 1,...,m
\end{eqnarray}
where $\forall{i}\in[m]$, $A_i\in\Rnxn$ is w.l.o.g. symmetric and $b_i\in\mathbb{R}$.  

\begin{definition}[$\epsilon$-approximated solution]
Given an instance of SDP of the form \eqref{SDP}, a matrix $X\in\Rnxn$ will be called an $\epsilon$-approximated solution if $X$ satisfies:
\begin{enumerate}
\item $A_i\bullet{X} \geq b_i - \epsilon \quad \forall{i\in[m]}$
\item $X \succeq -\epsilon\ident$
\end{enumerate}
\end{definition}

The main result of this paper is stated in the following theorem.
\begin{theorem}\label{MainThr}
There exists an algorithm that given $\epsilon > 0$ and an instance of the form \eqref{SDP} such that $\forall{i\in{[m]}}, \Vert{A_i}\Vert_F \leq 1, \vert{b_i}\vert \leq 1$ and there exists a feasible solution $X^*$ such that $\Vert{X^*}\Vert_F \leq 1$, returns an $\epsilon$-approximated solution with probability at least $1/2$.\\
The running time of the algorithm is $O\left({\frac{m\log{m}}{\epsilon^2} + \frac{n^2\log{m}\log{n}}{\epsilon^{2.5}}}\right)$.
\end{theorem}

Our upper bound is completed by the following lower bound that states that the running time of our algorithm is nearly optimal.
\begin{theorem}\label{LowerBoundThr}
Given an instance of the form \eqref{SDP} such that $\forall{i\in{[m]}}$ $\Vert{A_i}\Vert_F \leq 1$, $\vert{b_i}\vert \leq 1$, any algorithm that with probability at least $1/2$  does the following: either finds a matrix $X$ such that $X$ is an $\epsilon$-approximated solution and $\Vert{X}\Vert_F \leq 1$, or declares that no such matrix could be found, has running time at least $\Omega(\epsilon^{-2}(m+n^2))$.
\end{theorem}

\section{Preliminaries}

Denote the following sets: 
\begin{eqnarray*}
&&\ballf = \lbrace{X\in{\mathbb{R}^{n\times{}n}} \,|\, \Vert{X}\Vert_F \leq 1}\rbrace \\
&&\Delta_{m+1} = \lbrace{p\in\mathbb{R}^m \, | \, \forall{i}\in{[m]} \, p_i \geq 0, \sumim{}p_i \leq 1}\rbrace \\
&&\sdpball = \lbrace{X \in\Rnxn \, | \, X \succeq{0}, \textrm{Tr}(X) \leq 1}\rbrace
\end{eqnarray*}

We consider the following concave-convex problem  
\begin{eqnarray}\label{MaxMinSDP}
\max_{X\in\ballf}\quad\min_{p\in\Delta_{m+1}, Z\in\sdpball}\sumim{}p_i(A_i\bullet{}X-b_i) + Z\bullet{}X
\end{eqnarray}

The following claim establishes that in order to approximate \eqref{SDP} it suffices to approximate \eqref{MaxMinSDP}.

\begin{claim}
Given a feasible SDP instance of the form \eqref{SDP} let $X\in\ballf$ be such that
\begin{eqnarray*}
\min_{p\in\Delta_{m+1}, Z\in\sdpball}\sumim{}p_i(A_i\bullet{}X-b_i) + Z\bullet{}X \geq - \epsilon
\end{eqnarray*}
Then $X$ is an $\epsilon$-approximated solution.
\end{claim}

\begin{proof}
Define $\textrm{Val}(X) = \min_{p\in\Delta_{m+1}, Z\in\sdpball}\sumim{}p_i(A_i\bullet{}X-b_i) + Z\bullet{}X$. For all $i\in{[m]}$ it holds by setting the dual variables to $p_i = 1$, $p_j=0$ $\forall{i\neq j}$ and $Z=\textbf{0}_{n\times{n}}$ that
\begin{eqnarray*}
A_i\bullet{X}-b_i \geq \textrm{Val}(X) \geq -\epsilon
\end{eqnarray*}
Also, for any vector $v\in{\mathbb{R}^n}$ such that $\Vert{v}\Vert_2 \leq 1$ we set the dual variables to $p_i =0$ $\forall{i}$ and $Z = vv^{\top}$ and thus is holds that
\begin{eqnarray*}
v^{\top}Xv \geq \textrm{Val}(X) \geq -\epsilon
\end{eqnarray*}
which implies that $X \geq -\epsilon\textbf{I}$.
\end{proof}

\section{The Algorithm}

In this section we present our algorithm that approximates the max-min objective in \eqref{MaxMinSDP} up to a desired additive factor of $\epsilon$. Our algorithm can be viewed as a primal-dual algorithm that works in iterations, on each iteration performing a primal improvement step and a dual one. For this task we make use of online convex optimization algorithms which are known to be useful for solving concave-convex problems. \\
Consider the function $\mathcal{L}:\ballf\times\Delta_{m+1}\times\sdpball\rightarrow\mathbb{R}$ given by
\begin{eqnarray*}
\mathcal{L}(X,p,Z) = \sumim{}p_i(A_i\bullet{}X - b_i) + Z\bullet{}X
\end{eqnarray*}
The primal variable $X$ is updated by an  online stochastic gradient ascent algorithm which updates $X$ by
\begin{eqnarray*}
X_{t+1} \gets X_t + \eta\tilde{\nabla}_t
\end{eqnarray*}
where $\tilde{\nabla}_t$ is an unbiased estimator for the derivative of $\mathcal{L}(X,p,Z)$ with respect to the variable $X$, that is $\mathbb{E}[\tilde{\nabla}_t | p,Z] = \sumim{}p_iA_i+Z$. the parameter $\eta$ is the step size. Note that after such an update the point $X_{t+1}$ may be outside of the set $\ballf$ and we need to project it back to the feasible set which requires only to normalize the frobenius norm. Since we assume that the matrices $A_i$ are symmetric, then the primal variable $X$ is also always a symmetric matrix.\\
The dual variable $p$ which imposes weights over the constraints is updated by a variant of the well known multiplicative weights (MW) algorithm which performs the following updates:
\begin{eqnarray*}
w_{t+1} \gets w_te^{-\eta(A_i\bullet{}X-b_i)} ,\qquad p_{t+1} \gets \frac{w_{t+1}}{\Vert{w_t}\Vert_1}
\end{eqnarray*}
where $w$ is the vector of weights prior to the normalization to have $l_1$ norm equals 1. This update increases the weight of constraints the are not satisfied well by the current primal solution $X_t$.\\
The MW algorithm produces vectors $p_t$ which lie in the simplex, that is $\sumim{}p_t(i) = 1$. In our case we are interested that the sum of entries in $p_t$ may be less then 1. We enable this by artificially adding an additional constraint to the sdp instance in the form $\textbf{0}_{n\times{}n}\bullet{X} \geq 0$. And run the MW algorithm with dimension $m+1$. By the MW update rule, the size of the entry $p_{m+1}$ is fixed on all iteration and its entire purpose is to allow the sum of the first $m$ entries to be less than 1. The added constraint is of course always satisfied and thus it does not affect the optimization. \\
An additional issue with the MW updates is that it requires to compute on each iteration the products $A_i\bullet{X_t}$ for all $i\in{[m]}$ which takes linear time in the number of entries in the sdp instance. We overcome this issue by only sampling these products instead of using exact computation. Given the matrix $X$ we estimate the product $A_i\bullet{X}$ by
\begin{eqnarray*}
\tilde{v}_i \gets \frac{A_i(j,l)\Vert{X}\Vert_F^2}{X(j,l)} \quad \textrm{with probability} \quad \frac{X(j,l)^2}{\Vert{X}\Vert_F^2}
\end{eqnarray*}
It holds that $\mathbb{E}[\tilde{v}_i | X] = A_i\bullet{X}$. \\
On the down side the estimates $v_i$ are unbounded which is important to get high probability concentration guarantees. We overcome this difficulty by clipping these estimates by taking $v_i \gets \max\lbrace{\min\lbrace{\tilde{v}_i, \eta^{-1}}\rbrace, -\eta^{-1}}\rbrace$. Note that $v_i$ is no longer an unbiased estimator of $A_i\bullet{X}$, however the resulting bias is of the order of $\epsilon$ and thus does not hurt our analysis. Since the values $v_i$ may still be large we use the variance of these variables to get better concentration guarantees. It holds that
\begin{eqnarray*}
\mathbb{E}[v_i^2 | X] \leq \mathbb{E}[\tilde{v}_i^2 | X] = \Vert{A_i}\Vert_F^2\Vert{X}\Vert_F^2
\end{eqnarray*}
Finally the dual variable $Z$, unlike the variables $X,p$ which are updated incrementally, is always locally-optimized by choosing
\begin{eqnarray*}
Z \gets \min_{M\in\sdpball}M\bullet{X}
\end{eqnarray*}
Here we note that in case $X$ is not PSD then without loss of generality $Z$ is always a rank one matrix $zz^{\top}$ such that $z$ is an eigenvector of $X$ corresponding to the most negative eigenvalue of $X$. In case $X$ is PSD then $Z = \textbf{0}_{n\times{}n}$. In any case $\Vert{Z}\Vert_F \leq 1$. $Z$ could be approximated quite fast using an eigenvalue algorithm such as the Lanczos method. It will suffice to find a matrix $Z$ such that the product $Z\bullet{X}$ is $O(\epsilon)$ far from the true minimum. \\
Finally the algorithm returns the average of all primal iterates.

\begin{algorithm}[H]
\caption{SublinearSDP}
\label{Algorithm1}
\begin{algorithmic}[1]
\STATE Input: $\epsilon > 0$, $A_i\in{\mathbb{R}^{n\times{n}}}$, $b_i\in\mathbb{R}$ for $i\in{[m]}$.
\STATE Let $T \gets 20^2\sqrt{40}\epsilon^{-2}\log{m}$, $\eta \gets \sqrt{\frac{40\log{m}}{T}}$, $\epsilon' \gets \epsilon/4$.
\STATE Let $Y_1 \gets ֿ\textbf{0}_{n\times{n}}$, $w_1 \gets \textbf{1}_{m}$.
\STATE Let $A_{m+1} = \textbf{0}_{n\times{n}}$, $b_{m+1} = 0$.
\FOR{$t = 1$ to T}
\STATE $X_t \gets Y_t / \max\lbrace{1, \Vert{Y_t}\Vert_F}\rbrace$.
\STATE $p_t \gets \frac{w_t}{\Vert{}w_t\Vert{}_1+1}$.
\STATE $Z_t \gets Z\in\Rnxn$ s.t. $Z\bullet{}X_t \leq \min_{Z\in{\sdpball}}Z\bullet{X_t} + \epsilonֿ'$.
\STATE $i_t \gets i\in[m]$ w.p. $p_t(i)$ and $i_t \gets m+1$ w.p. $1-\sumim{}p_t(i)$.
\STATE $Y_{t+1}\gets Y_t + \frac{1}{\sqrt{2T}}(A_{i_t} + Z_t)$
\STATE Choose $(j_t,l_t)\in{[n]\times{}[n]}$ by $(j_t,l_t) \gets (j,l)$ w.p. $X_t(j,l)^2/\Vert{X_t}\Vert_F^2$.
\FOR{$i\in{[m]}$}
\STATE $\tilde{v}_t \gets A_i(j_t,l_t)\Vert{X_t}\Vert^2/X_t(j_t,l_t) - b_i$
\STATE $v_t(i) \gets $clip$(\tilde{v}_t(i), 1/\eta)$
\STATE $w_{t+1}(i) \gets w_t(i)(1 - \eta{}v_t(i) + \eta^2v_t(i)^2)$
\ENDFOR
\ENDFOR
\RETURN $\bar{X} = \frac{1}{T}\sum_{t}X_t$
\end{algorithmic}
\end{algorithm}

\section{Analysis}

The following lemma gives a bound on the regret of the MW algorithm (line 15), suitable for the case in which the losses are random variables with bounded variance. For a proof see \cite{CHW_ARXIV} Lemma 2.3.

\begin{lemma}\label{lemma:mw}
The MW algorithm satisfies
\begin{eqnarray*}
\sum_{t\in{[T]}}p_t^{\top}q_t \leq \min_{i\in{[m]}}\sum_{t\in{[T]}}\max\{q_t(i), -\frac{1}{\eta}\} + \frac{\log{m}}{\eta} + \eta\sum_{t\in{[t]}}p_t^{\top}q_t^2
\end{eqnarray*}
\end{lemma}

The following lemma gives concentration bounds on our random variables from their expectations. The proof is given in the appendix.

\begin{lemma}\label{lemma:concentration}
For $1/4 \geq \eta \geq \sqrt{\frac{40\log{m}}{T}}$, with probability at least $1-O(1/m)$, it holds that
\begin{eqnarray*}
(i) & \max_{i\in{[m]}}\vert{\sum_{t\in{[T]}}\left({A_i\bullet{}X_t - b_i}\right) - v_t(i)}\vert \leq 3\eta{}T \\ 
(ii) & \left|\sum_{t\in{[T]}}\left({A_{i_t}\bullet{}X_t - b_{i_t}}\right)- \sum_{t\in{[T]}}p_t^{\top}v_t\right| \leq 4\eta{}T
\end{eqnarray*}
\end{lemma}

The following Lemma gives a regret bound on the online gradient ascent algorithm used in our algorithm (line 10). For a proof see \cite{Zinkevich}.

\begin{lemma}\label{lemma:ogd}
Consider matrices $M_1,...,M_T\in{\mathbb{R}^{n\times{}n}}$ such that for all $i\in{[m]}$ $\Vert{M_i}\Vert_F \leq \rho$. Let $X_0 = \mathbf{0}_{n\times{}n}$ and for all $t \geq 1$ let $Y_{t+1} = X_t + \frac{\sqrt{2}}{\rho\sqrt{T}}M_t$ and $X_{t+1} = \min_{X\in\ballf}\Vert{Y_{t+1}-X}\Vert_F$. Then
\begin{eqnarray*}
\max_{X\in{\ballf}}\sum_{t\in{[T]}}M_t\bullet{}X - \sum_{t\in{[T]}}M_t\bullet{}X_t \leq 2\rho\sqrt{2T}
\end{eqnarray*}
\end{lemma}

We are now ready to prove our main theorem, theorem \ref{MainThr}.

\begin{proof}
By applying lemma \ref{lemma:ogd} with parameters $M_t = A_{i_t} + Z_t$ and $\rho = 2$ we get

\begin{eqnarray*}
\max_{x\in{\ballf}}\sum_{t\in{[T]}}\left({A_{i_t} + Z_t}\right)\bullet{}X - 
\sum_{t\in{[T]}}\left({A_{i_t} + Z_t}\right)\bullet{}X_t \leq 4\sqrt{2T}
\end{eqnarray*}
Adding and subtracting $\sumtT{}b_{i_t}$ gives

\begin{eqnarray*}
\max_{x\in{\ballf}}\sum_{t\in{[T]}}\left({A_{i_t}\bullet{}X - b_{i_t} + Z_t\bullet{}X}\right) - 
\sum_{t\in{[T]}}\left({A_{i_t}\bullet{}X_t - b_{i_t} + Z_t\bullet{}X_t}\right) \leq 4\sqrt{2T}
\end{eqnarray*}
Since we assume that there exists a feasible solution $X^*\in\ballf$ we have that

\begin{eqnarray}\label{mt:eq1}
\sum_{t\in{[T]}}\left({A_{i_t}\bullet{}X_t - b_{i_t} + Z_t\bullet{}X_t}\right) \geq -4\sqrt{2T}
\end{eqnarray}
Turning to the MW part of the algorithm, by lemma \ref{lemma:mw}, and using the clipping of $v_t(i)$ we have

\begin{eqnarray*}
\sum_{t\in{[T]}}p_t^{\top}v_t \leq \min_{p\in{\Delta_{m+1}}}\sum_{t\in{[t]}}p^{\top}v_t + (\log{m})/\eta + \eta\sum_{t\in{[T]}}p_t^{\top}v_t^2
\end{eqnarray*}
By lemma \ref{lemma:concentration} (i), with high probability and for any $i\in{[m]}$,

\begin{eqnarray*}
\sum_{t\in{[T]}}v_t(i) \leq \sum_{t\in{[T]}}A_i\bullet{}X_t - b_i + 3\eta{}T
\end{eqnarray*}
Thus with high probability it holds that
\begin{eqnarray*}
\sum_{t\in{[T]}}p_t^{\top}v_t \leq \min_{p\in{\Delta_{m+1}}}\sum_{t\in{[T]}}\sumim{}p_i\left({A_i\bullet{}X_t -b_i}\right) + (\log{m})/\eta + \eta\sum_{t\in{[T]}}p_t^{\top}v_t^2 + 3\eta{}T
\end{eqnarray*}
Applying lemma \ref{lemma:concentration} (ii) we get that with high probability

\begin{eqnarray*}
&\sumtT\left({A_{i_t}\bullet{X_t} - b_{i_t}}\right) \leq \\ &\min_{p\in{\Delta_{m+1}}}\sum_{t\in{[T]}}\sumim{}p_i\left({A_i\bullet{}X_t -b_i}\right) + (\log{m})/\eta + \eta\sum_{t\in{[T]}}p_t^{\top}v_t^2 + 7\eta{}T
\end{eqnarray*}
Adding $\sumtT{}Z_t\bullet{}X_t$ to both sides of the inequality and using \eqref{mt:eq1} yields

\begin{eqnarray}\label{mt:eq2}
\min_{p\in{\Delta_{m+1}}}\sum_{t\in{[T]}}\left({\sumim{}p_i\left({A_i\bullet{}X_t -b_i}\right)+Z_t\bullet{}X_t}\right)
\geq \nonumber \\
-4\sqrt{2T} - (\log{m})/\eta - \eta\sum_{t\in{[T]}}p_t^{\top}v_t^2 - 7\eta{}T
\end{eqnarray}
It holds that 

\begin{eqnarray*}
\sumtT{}Z_t\bullet{X_t} \leq \sumtT{}\min_{Z\in\sdpball}\left({Z \bullet{X_t} + \epsilon'}\right) \leq 
\min_{Z\in\sdpball}\sumtT{}\left({Z\bullet{}X_t + \epsilon'}\right)
\end{eqnarray*} 
Plugging the last inequality into \eqref{mt:eq2} gives

\begin{eqnarray}\label{mt:eq3}
\min_{p\in{\Delta_{m+1}}, Z\in{\sdpball}}\sum_{t\in{[T]}}\left({\sumim{}p_i\left({A_i\bullet{}X_t -b_i}\right)+Z\bullet{}X_t}\right)
\geq \nonumber \\
-4\sqrt{2T} - (\log{m})/\eta - \eta\sum_{t\in{[T]}}p_t^{\top}v_t^2 - 7\eta{}T - \epsilon'T
\end{eqnarray}
By a simple Markov inequality argument it holds that w.p. at least $3/4$,
\begin{eqnarray*}
\sum_{t\in{[T]}}p_t^{\top}v_t^2 \leq 4T
\end{eqnarray*}
Plugging this bound into \eqref{mt:eq3} and dividing through by $T$ gives with probability at least $1/2$

\begin{eqnarray*}
\min_{p\in{\Delta_{m+1}}, Z\in{\sdpball}}\sumim{}p_i\left({A_i\bullet{}\bar{X} -b_i}\right)+Z\bullet{}\bar{X}
\geq \\
-\frac{4\sqrt{2}}{\sqrt{T}} - (\log{m})/(\eta{}T) - 11\eta - \epsilon'
\end{eqnarray*}

The theorem follows from plugging the values of $T$, $\eta$ and $\epsilon'$.
\end{proof}

\noindent The algorithm performs $O(\epsilon^{-2}\log{m})$ iterations. Each iterations includes a primal gradient update step which takes $O(n^2)$ time to compute, updating the distribution over constrains using a single sample per constraint which takes $O(m)$ time and computing a single eigenvalue up to an $O(\epsilon)$ approximation which using the lanczos method takes at most $O(\frac{n^2\log{n}}{\sqrt{\epsilon}})$ time (see \cite{EigenvaluesApprox} theorem 3.2) . Overall the running time is as stated in theorem \ref{MainThr}. 

\section{Lower Bound}

In this section we prove Theorem \ref{LowerBoundThr}. Our proof relies on an information theoretic argument as follows: We show that it is possible to generate two random SDP instances such that one is feasible and the other one is far from being feasible. We show that these two random instances differ only by a single entry chosen also at random. Any successful algorithm must distinguish between these two instances and thus must read the single distinguishing entry which requires any algorithm to read a constant factor of the total number of relevant entries in order to succeed with constant probability. \\
We split our random construction into the following two lemmas.

\begin{lemma}
Under the conditions stated in Theorem \ref{LowerBoundThr}, any successful algorithm must read $\Omega\left({\frac{m}{\epsilon^2}}\right)$ entries from the input.
\end{lemma}

\begin{proof}
Assume that $n \geq \frac{1}{\epsilon}$. Consider the following random instance. With probability $1/2$ each of the constraint matrices $A_i$ has a single randomly chosen entry $(i,j)\in{[\frac{1}{2\epsilon}]\times{}[\frac{1}{2\epsilon}]}$ that equals $\sqrt{1-\zeta^2\left({\frac{1}{4\epsilon^2}-1}\right)}$ and all other entries take random values from the interval $[0,\zeta]$ (the goal of these values is to prevent a sparse representation of the input). With the remaining probability of $1/2$, all constraint matrices except one are exactly as before except for a single constraint matrix (chosen at random uniformly) that has all of its entries chosen at random from $[0,\zeta]$. In both cases for each constraint matrix $A_i, i\in{[m]}$ it holds that $\Vert{A_i}\Vert_F \leq 1$. \\
In the second case it clearly holds that for all $X\in\ballf$,

\begin{eqnarray*}
\min_{i\in[m]}A_i\bullet{X} \leq  \sqrt{\frac{1}{4\epsilon^2}\cdot\zeta^2} = \frac{\zeta}{2\epsilon}
\end{eqnarray*}
In the first case we can construct a solution matrix $X^*$ has follows: for each $(i,j)\in{[\frac{1}{2\epsilon}]\times{}[\frac{1}{2\epsilon}]}, X^*(i,j) = 2\epsilon$ and 0 elsewhere. Clearly $X^*$ is positive semi definite (since it is a symmetric rank-one matrix) and $\Vert{X}\Vert_F = 1$. For each $i\in{[m]}$ it holds that 

\begin{eqnarray*}
A_i\bullet{X^*} \geq 2\epsilon\cdot \sqrt{1-\zeta^2\left({\frac{1}{4\epsilon^2}-1}\right)}
\end{eqnarray*}
By choosing $\zeta = \epsilon^2$ and in both cases $b_i = 1.6\epsilon \, \forall{i\in[m]}$ we have that in the first case
\begin{eqnarray*}
\min_{i\in[m]}A_i\bullet{X^*}-b_i & \geq &
2\epsilon\sqrt{1-\epsilon^4\left({\frac{1}{4\epsilon^2}-1}\right)} - 1.6\epsilon \\
& > &\left({\sqrt{3}-1.6}\right)\epsilon > 0.1\epsilon
\end{eqnarray*}
In the second case, for all $X\in\ballf$ it holds that,

\begin{eqnarray*}
\min_{i\in[m]}A_i\bullet{X}-b_i \leq \frac{\epsilon}{2} - 1.6\epsilon = -1.1\epsilon
\end{eqnarray*}

Thus the first instance is feasible while the second one does not admit an $\epsilon$-approximated solution and the two instances differ by a single randomly chosen entry.
\end{proof}

\begin{lemma}
Under the conditions stated in Theorem \ref{LowerBoundThr}, any successful algorithm must read $\Omega\left({\frac{n^2}{\epsilon^2}}\right)$ entries from the input.
\end{lemma}

\begin{proof}
The proof follows the lines of the previous proof. Assume that $m \geq \frac{1}{16\epsilon^2}$, $\epsilon \geq \frac{1}{\sqrt{n}}$ and that $n$ is even. Let $p,q\in\mathbb{N}^n$ be two random permutations over the integers $1..n/2$ and finally set $q_i = \frac{n}{2}+q_i$. Consider the following random instance composed of $\frac{1}{16\epsilon^2}$ constraint matrices $A_i, i\in{[\frac{1}{16\epsilon^2}]}$. With probability $1/2$ for each $A_i$ we set the entry $A_i(p_i,q_i)$ to equal $\sqrt{1-\zeta^2\left({n^2-1}\right)}$ and all other entries in $A_i$ are sampled uniformly from $[0,\zeta]$. With the other $1/2$ probability, all matrices are as before with the difference that we randomly pick a matrix $A_j$, $j\in[m]$ and set $A_j(p_j,q_j)$ to a value sampled uniformly from  $[0,\zeta]$. In both cases it holds that $\Vert{A_i}\Vert_F \leq 1$ for all $i\in[m]$. \\
In the second case it holds for all $X\in\ballf$ that,

\begin{eqnarray*}
\min_{i\in[m]}A_i\bullet{X} \leq  n\zeta
\end{eqnarray*}
In the first case we construct a solution $X^*$ as follows. For every $i\in[m]$ we define a matrix $X_i^*$ such that $X_i^*(p_i,q_i) = X_i^*(q_i,p_i) = X_i^*(p_i,p_i) = X_i^*(q_i,q_i) = 2\epsilon$ and $X_i^*$ is zero elsewhere. Finally we take $X^* = \sumim{}X_i^*$. \\
Notice that $X^*$ is the sum of symmetric rank-one matrices and thus it is positive semidefinite. \\
Since $p,q$ are both permutations over disjoint sets we have that for every $i,j\in[n]\times[n]$ it holds that $|X^*(i,j)| \leq 2\epsilon$ and thus $\Vert{X^*}\Vert_F^2 \leq \frac{1}{16\epsilon^2}\cdot{}4\cdot{}4\epsilon^2 = 1$. \\
By construction it holds for every $i\in[m]$ that

\begin{eqnarray*}
A_i\bullet{X^*} \geq 
2\epsilon\sqrt{1-\zeta^2\left({n^2-1}\right)}
\end{eqnarray*}
By choosing $\zeta = \frac{\epsilon}{2n}$ and in both cases $b_i = 1.6\epsilon \, \forall{i\in[m]}$ we have that in the first case
\begin{eqnarray*}
\min_{i\in[m]}A_i\bullet{X^*}-b_i & \geq &
2\epsilon\sqrt{1-\frac{\epsilon^2}{4n^2}\left({n^2-1}\right)} - 1.6\epsilon \\
& > &\left({\sqrt{3}-1.6}\right)\epsilon > 0.1\epsilon
\end{eqnarray*}
In the second case, for all $X\in\ballf$ it holds that,

\begin{eqnarray*}
\min_{i\in[m]}A_i\bullet{X}-b_i \leq \frac{\epsilon}{2} - 1.6\epsilon \leq -1.1\epsilon
\end{eqnarray*}
Thus as before, the first instance is feasible while the second one does not have an $\epsilon$ additive approximated solution and the two instances differ by a single entry. Notice however that unlike the previous lemma, in this case because of the nature of our random construction, after reading $k$ matrices it is suffices for an algorithm searching for the distinguishing entry, to only search $\left({\frac{n}{2}-k}\right)^2$ entries in the next matrix. Nevertheless, by plugging the values of $m$ and the lower bound on $\epsilon$ we get that $\left({\frac{n}{2}-m}\right)^2 \geq \frac{n^2}{4} -\frac{n}{16\epsilon^2} \geq \frac{3n^2}{16}$ and thus any algorithm must still read an order of $n^2$ entries from each matrix.

\end{proof}

\appendix

\section{Martingale and concentration lemmas}

We first prove a lemma on the expectation of clipped random variables.

\begin{lemma} \label{lem:clip}
Let $X$ be a random variable, let $\bar{X} = \clip(X, C) = \min\{C, \max\{-C, X\}\}$
and assume that $ |\E[X]| \leq C/2 $ for some $C>0$.
Then
\[
\left| \E[\bar{X}] - \E[X] \right| \le \frac{2}{C} \var[X] .
\]
\end{lemma}

\begin{proof}
As a first step, note that for $x>C$ we have $x-\E[X] \ge C/2$, so that
\[
C(x-C) \le 2(x-\E[X])(x-C) \le 2(x-\E[X])^2 .
\]
Hence, we obtain
\begin{align*}
\E[X] - \E[\bar{X}]
       & =   \int_{x<-C} (x+C) d\mu_X + \int_{x>C} (x-C) d\mu_X
    \\ & \le \int_{x>C} (x-C) d\mu_X
    \\ & \le \frac{2}{C} \int_{x>C} (x-\E[X])^2 d\mu_X
    \\ & \le \frac{2}{C} \var[X] .
\end{align*}
Similarly one can prove that $\E[X] - \E[\bar{X}] \ge -2\var[X]/C$, and the result follows.
\end{proof}

\noindent The following lemmas are used to prove lemma \ref{lemma:concentration}. \\
In the following we assume only that $v_t(i) = \clip(\vtil_t(i),1/\eta)$ is the clipping of a random variable $\vtil_t(i)$, the conditional variance of $\vtil_t(i)$ is at most one ($\var[\vtil_t(i) \, | \, X_t] \leq 1$) and we use the notation $\mu_t(i) = \E[\vtil_t(i) \, | \, X_{t}] = A_i\bullet{}X_t^{\top} - b_i$. We also assume that the expectations of $\vtil_t(i)$ are bounded in absolute value by a constant $|\mu_t(i)| = |A_i\bullet{}X_t - b_i|\leq C$, such that $2 \le 2C \le 1/\eta$.\\
Both lemmas are based on an application of Freedman's inequality which is a Bernstein-like concentration inequality for martingales which we now state:

\begin{lemma}[Freedman's inequality]
Let $\xi_1,...,\xi_T$ be a martingale difference sequence with respect to a certain filtration $\lbrace{S_t}\rbrace$, that is $\mathbb{E}[\xi_t \, | \, S_t] = 0$ for every $t$. Assume also that for every $t$ it holds that $\vert{\xi_t}\vert \leq V$ and $\mathbb{E}[\xi_t^2 \, | \, S_t] \leq s$. Then
\begin{eqnarray*}
P\left({\vert{\sum_{t=1}^T\xi_t}\vert \geq \epsilon}\right) \leq 2\exp\left({-\frac{\epsilon^2/2}{Ts + V\epsilon/3}}\right)
\end{eqnarray*}
\end{lemma}

\begin{lemma} \label{lem:conc1_generic}
For $\frac{1}{2C} \geq \eta \geq\sqrt{\frac{4\log{(2m^2)}}{T}}$ it holds with probability at least $1-\frac{1}{m}$ that
\begin{eqnarray*}
\max_{i\in[m]}\vert{\sumtT{}v_t(i) - \mu_t(i)}\vert \leq 3\eta{}T
\end{eqnarray*}
\end{lemma}

\begin{proof}
Given $i\in[m]$, consider the martingale difference sequence $\xi_t^i = v_t(i) - \mathbb{E}[v_t(i)]$ with respect to the filtration $S_t = (X_t)$. \\
 It holds that for all $t$, $\vert{\xi_t^i}\vert \leq \frac{2}{\eta}$ and $\mathbb{E}[(\xi_t^i)^2 \, | \, S_t] \leq 1$. Applying Freedman's inequality we get
\begin{eqnarray*}
P\left({\vert{\sum_{t=1}^T\xi_t}\vert \geq \eta{}T}\right) &\leq & 
2\exp\left({-\frac{\eta^2T^2/2}{T+(2/\eta)\eta{}T/3}}\right) \\
&\leq & 2\exp\left({-\eta^2T/4}\right)
\end{eqnarray*}
Using lemma \ref{lem:clip} the fact that $v_t(i)$ is the clipping of $\tilde{v}_t(i)$ and the triangle inequality we have,

\begin{eqnarray*}
P\left({\vert{\sumtT{}v_t(i) - \mu_t(i)}\vert \geq 3\eta{}T}\right) \leq 2\exp\left({-\eta^2T/4}\right)
\end{eqnarray*}
Thus for $\eta \geq \sqrt{\frac{4\log{(2m^2)}}{T}}$ we have that with probability at least $1-\frac{1}{m^2}$,
\begin{eqnarray*}
\vert{\sumtT{}v_t(i) - \mu_t(i)}\vert \leq 3\eta{}T
\end{eqnarray*}
The lemma follows from taking the union bound over all $i\in[m]$.
\end{proof}

\begin{lemma}\label{lem:conc2_generic}
For $\frac{1}{2C} \geq \eta \geq\sqrt{\frac{4\log{(2m^2)}}{T}}$ it holds with probability at least $1-\frac{1}{m}$ that
\begin{eqnarray*}
\abs{ \sumt p_t \tr v_t - \sumt p_t \tr \mu_t }  \le 3 \eta T  .
\end{eqnarray*}
\end{lemma}

\begin{proof}
This Lemma is proven in essentially the same manner as Lemma \ref{lem:conc1_generic}, and proven below for completeness.\\
Consider the martingale difference sequence $\xi_t = p_t^{\top}v_t - \mathbb{E}[p_t^{\top}v_t]$ with respect to the filtration $S_t = (X_t, p_t)$. \\
It holds for all $t$ that $\vert{\xi_t}\vert \leq \frac{2}{\eta}$. Also by convexity it holds that $\mathbb{E}[\xi_t^2 \, | \, S_t] = \mathbb{E}[(p_t^{\top}v_t)^2 \, | \, S_t] \leq \sumim{}p_t(i)\mathbb{E}[v_t(i)^2 \, | \, S_t] \leq 1$.\\
Applying Freedman's inequality we have,
\begin{eqnarray*}
P\left({\vert{\sum_{t=1}^T\xi_t}\vert \geq \eta{}T}\right) \leq & 2\exp\left({-\eta^2T/4}\right)
\end{eqnarray*}
Using lemma \ref{lem:clip} the fact that $v_t(i)$ is the clipping of $\tilde{v}_t(i)$ and the triangle inequality we have,

\begin{eqnarray*}
P\left({\vert{\sumtT{}p_t^{\top}v_t - p_t^{\top}\mu_t}\vert \geq 3\eta{}T}\right) \leq 2\exp\left({-\eta^2T/4}\right)
\end{eqnarray*}
Thus for $\eta \geq \sqrt{\frac{4\log{(2m^2)}}{T}}$ the lemma follows.
\end{proof}

\begin{lemma} \label{lem:conc3_generic}
For $\frac{1}{2C} \geq \eta \geq \sqrt{\frac{10C\log{(2m)}}{T}}$, with probability at least $1-1/m$,
\begin{eqnarray*}
\abs{ \sumt \mu_t(i_t) - \sumt p_t \tr \mu_t }  \le \eta T  .
\end{eqnarray*}
\end{lemma}

\begin{proof}
Consider the martingale difference $\xi_t = \mu_t(i_t) -  p_t \tr \mu_t $, where now $\mu_t$ is a constant vector and $i_t$ is the random variable, and consider the filtration given by $S_t = (X_t, p_t)$.\\
The expectation of $\mu_t(i_t)  $, conditioning on $S_t$ with respect to the random choice of the index $i_t$, is $p_t \tr \mu_t$. Hence $\E_t[\xi_t \, | \, S_t] = 0$. \\
It holds that $\vert{\xi_t}\vert \leq |\mu_t(i)| + | p_t \tr \mu_t| \leq 2C$.  Also
$\E[\xi_t^2] = \E[ (\mu_t(i)  - p_t \tr \mu_t)^2 ] \leq 2 \E[ \mu_t(i)^2] + 2 (p_t \tr \mu_t)^2 \leq 4C^2$. \\
Applying Freedman's inequality gives,
\begin{eqnarray*}
P\left({\vert{\sum_{t=1}^T\xi_t}\vert \geq \eta{}T}\right) &\leq & 
2\exp\left({-\frac{\eta^2T^2/2}{4C^2T+2C\eta{}T/3}}\right) \\
&\leq & 2\exp\left({-\eta^2T/(10C^2)}\right)
\end{eqnarray*}
where for the last inequality we use $C \geq 1$ and $\eta \leq \frac{1}{C}$. \\
Thus for $\eta \geq \sqrt{\frac{10C\log{(2m)}}{T}}$ the lemma follows.
\end{proof}

Setting $C=2$ and $\eta = \sqrt{\frac{40\log{m}}{T}}$ lemma \ref{lem:conc1_generic} yields part $(i)$ of lemma \ref{lemma:concentration} and combining combining lemmas \ref{lem:conc2_generic} and \ref{lem:conc3_generic} via the triangle inequity yields part $(ii)$ of lemma \ref{lemma:concentration}.

\bibliography{bib}{}
\bibliographystyle{unsrt}

\end{document}